\documentclass[a4paper,12pt,reqno]{amsart} 

\usepackage[utf8]{inputenc} 
\usepackage[T1]{fontenc} 
\usepackage{amsmath} 
\usepackage{graphicx} 
\usepackage{xcolor} 
\usepackage{tikz}
\usepackage{subcaption}
\usetikzlibrary{calc, fit}
\usetikzlibrary{shapes.geometric}
\usepackage{amssymb}
\usepackage{amsthm}
\usetikzlibrary{cd}
\usepackage{euscript}
\usepackage{authblk}
\usepackage{amsaddr}
\usepackage{hyperref}

\theoremstyle{definition}
\newtheorem{lemma}{Lemma}

\newtheorem*{thm}{Main Theorem}

\newcommand{\m}[1]{{\bf {#1}}}									
\newcommand{\lgc}[1]{\ensuremath{{\mathcal #1}}}						
\newcommand{\cls}[1]{\ensuremath{{\sf #1}}}						
\newcommand{\opr}[1]{\ensuremath{{\mathbb #1}}}					
\newcommand{\prp}[1]{{\sf #1}}

\renewcommand{\phi}{\varphi}
\newcommand{\fb}{\varphi_\m{B}}
\newcommand{\fc}{\varphi_\m{C}}
\newcommand{\gb}{\psi_\m{B}}
\newcommand{\gc}{\psi_\m{C}}
\newcommand{\emd}{\hookrightarrow}

\newcommand{\var}{\mathsf{var}}

\newcommand{\hm}{\opr{H}}

\newcommand{\pu}{\opr{P}_\mathsf{U}}
\newcommand{\sub}{\opr{S}}
\newcommand{\vr}{\opr{V}}

\newcommand{\hspu}{\hm\sub\pu}

\newcommand{\V}{\cls{V}}

\newcommand{\K}{\cls{K}}

\newcommand{\Vfsi}{{\cls{V}_{\mathrm{FSI}}}}

\newcommand{\C}{\cls{C}}
\newcommand{\M}{\cls{M}}
\newcommand{\LC}{\lgc{C}}
\newcommand{\LM}{\lgc{M}}

\newcommand{\bij}{\mathrel{\hookrightarrow\hspace{-1.8ex}\to}}

\newcommand{\tc}{\textrm{\textup{t}}}
\newcommand{\zr}{\textrm{\textup{f}}}

\makeatletter
\g@addto@macro{\endabstract}{\@setabstract}
\newcommand{\authorfootnotes}{\renewcommand\thefootnote{\@fnsymbol\c@footnote}}%
\makeatother

\newcommand\extrafootertext[1]{%
    \bgroup
    \renewcommand\thefootnote{\fnsymbol{footnote}}%
    \renewcommand\thempfootnote{\fnsymbol{mpfootnote}}%
    \footnotetext[0]{#1}%
    \egroup
}

\begin{document}
\begin{center}
  \large
  {\bf Revisiting Interpolation in Relevant Logics} \par \bigskip

  \normalsize
  \authorfootnotes
  Wesley Fussner\footnote{e-mail: \url{fussner@cs.cas.cz}}\textsuperscript{1}, Andrew Tedder\footnote{e-mail: \url{ajtedder.at@gmail.com}}\textsuperscript{2}\par  \bigskip

  \textsuperscript{1}Institute of Computer Science of the Czech Academy of Sciences \par
  \textsuperscript{2}Department of Philosophy I, Ruhr University Bochum\par 
  \footnotetext[0]{W. Fussner's research was supported by the Czech Science Foundation project 25-18306M, INTERACT. A. Tedder's research was supported by the DFG project TE 1611/1-1.}
  \setcounter{footnote}{0}

\end{center}
\begin{abstract}
There are exactly two maximal schematic extensions of the relevant logic $\lgc{R}$ with the variable sharing property. We establish that one of them has a strong form of interpolation for deducibility, thereby giving an example of a well-known relevant logic with interpolation.
\end{abstract}

\medskip
Urquhart showed in \cite{Urq93} that the relevant logic $\lgc{R}$ lacks the Craig interpolation property, as formulated in terms of the usual implication connective $\to$, and that the same holds for many of $\lgc{R}$'s most prominent neighbors.\footnote{We'll use $\lgc{C}$alligraphic letters to denote logics, $\m{B}$oldface letters for particular algebras, and \textsf{S}ans-serif letters for classes of algebras and some properties, such as the \prp{VSP}.} Due to the breadth of Urquhart's study and $\lgc{R}$'s status as the default relevant logic, many logicians today remember these results by the slogan \emph{relevant logics don't have interpolation}. The purpose of this note is to show that this slogan is a misapprehension: We show that $\lgc{R}$ possesses schematic extensions that satisfy both the variable sharing property, often taken to be \emph{the} distinctive feature of relevant logics, as well as a strong form of interpolation with respect to Tarskian consequence.

To be sure, it has long been known that $\lgc{R}$ has extensions that satisfy various interpolation properties, particularly when the language is enriched with constants $\tc$ and $\zr$ for truth and falsity; see, e.g., \cite[p.~450]{Urq93}, \cite[Theorem~4.9]{MarMet2012}, and \cite[Proposition~5.3]{FSsemilin}. However, the extensions considered in the aforementioned studies are dubiously relevant insofar as they lack the following \emph{variable sharing property} (or \emph{\prp{VSP}}):
\begin{equation}
  \tag{\prp{VSP}}\label{eq:VSP}
  \parbox{\dimexpr\linewidth-6em}{%
    \strut
	If $\vdash\alpha\to\beta$, then $\var(\alpha)\cap\var(\beta)\neq\emptyset$.
    \strut
    }
\end{equation}
There is broad consensus in the literature that the \prp{VSP} is at least a necessary condition for a logic to count as relevant \cite{Stan25}, and it has been shown that an extension $\lgc{L}$ of $\lgc{R}$ has the \prp{VSP} if and only of $\lgc{L}$ is contained in one of the logics $\LC$ or $\LM$, which are respectively the logics of the crystal lattice $\m{C}$ and Belnap's \cite{Belnap59} model $\m{M}$ (see \cite{Swi99}). In particular, $\LC$ and $\LM$ are themselves the only maximal schematic extensions of $\lgc{R}$ with the \prp{VSP}. Our main contribution is that $\LC$ has the following \emph{Maehara interpolation property} (or \emph{\prp{MIP}}), a strong form of interpolation for deducibility:
\begin{equation}
  \tag{\prp{MIP}}\label{eq:MIP}
  \parbox{\dimexpr\linewidth-6em}{%
    \strut
	If $\var(\Sigma\cup\{\alpha\})\cap\var(\Gamma)\neq\emptyset$ and $\Sigma,\Gamma\vdash\alpha$, there exists a formula $	\delta$ such that $\var(\delta)\subseteq\var(\Sigma\cup\{\alpha\})\cap\var(\Gamma)$ and both 
	$\Gamma\vdash\delta$ and $\Sigma,\delta\vdash\alpha$.
    }
\end{equation}
This implies, in particular, that $\LC$ has the well-known \emph{deductive interpolation property}, obtained by taking $\Sigma=\emptyset$ in the definition of the \prp{MIP}.

All the results here follow promptly from established methods. We offer no new technical innovations that deepen these methods, but rather aim to point out their application to this context, which has been the subject of misunderstanding. We hope here to clear up this misunderstanding, while also publicizing the utility of the technical tools we invoke here.\\

\paragraph{\bf Algebraic preliminaries.} To make sense of our main theorem's proof, we first give some background definitions. For further information on algebraic preliminaries, please consult \cite{GalatosJipsenKowalskiOno2007}.

If $\K$ is any class of similar algebras, a \emph{span in $\K$} is a pair of embeddings $\langle\fb\colon\m{A}\emd\m{B},\fc\colon\m{A}\emd\m{C}\rangle$ between members of $\K$. If $\K'$ is a class of algebras in the same similarity type of $\K$ and $\K\subseteq\K'$, an \emph{amalgam in $\K'$} of a span $\langle\fb\colon\m{A}\emd\m{B},\fc\colon\m{A}\emd\m{C}\rangle$ is a pair $\langle\gb\colon\m{B}\emd\m{D},\gc\colon\m{C}\emd\m{D}\rangle$ of embeddings in $\K'$ such that $\gb\fb=\gc\fc$ (see Figure~\ref{fig:AP and TIP}, left). The class $\K$ has the \emph{amalgamation property} (or \emph{\prp{AP}}) if every span in $\K$ has an amalgam in $\K$. On the other hand, a class $\K$ has the \emph{transferable injections property} (or \emph{\prp{TIP}}) if whenever $\langle\fb\colon\m{A}\to\m{B},\fc\colon\m{A}\emd\m{C}\rangle$ is a pair of homomorphisms in $\K$ such that $\fc$ is an embedding, there exists an algebra $\m{D}\in\K$ and a pair of homomorphisms $\langle\gb\colon\m{B}\emd\m{D},\gc\colon\m{C}\to\m{D}\rangle$ such that $\gb$ is an embedding and $\gb\fb=\gc\fc$ (see Figure~\ref{fig:AP and TIP}, right).

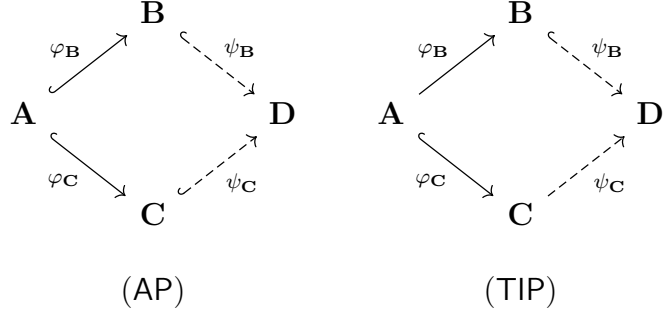
\begin{figure}[t]
\centering
\captionsetup{justification=centering}
\begin{tabular}{ccc}
\begin{tikzcd}
	& {\bf B} \\
	{\bf A} && {\bf D} \\
	& {\bf C}
	\arrow["{\fb}", hook, from=2-1, to=1-2]
	\arrow["{\gb}", dashed, hook, from=1-2, to=2-3]
	\arrow["{\fc}"', hook, from=2-1, to=3-2]
	\arrow["{\gc}"', dashed, hook, from=3-2, to=2-3]
\end{tikzcd}
& &
\begin{tikzcd}
	& {\bf B} \\
	{\bf A} && {\bf D} \\
	& {\bf C}
	\arrow["{\fb}", from=2-1, to=1-2]
	\arrow["{\gb}", dashed, hook, from=1-2, to=2-3]
	\arrow["{\fc}"', hook, from=2-1, to=3-2]
	\arrow["{\gc}"', dashed,  from=3-2, to=2-3]
\end{tikzcd}
\\ \\
(\prp{AP}) & & (\prp{TIP})
\end{tabular}
\caption{The commutative diagrams illustrating the \prp{AP} (left) and the \prp{TIP} (right).}
\label{fig:AP and TIP}
\end{figure}

A class $\K$ of algebras has the \emph{congruence extension property} (or \emph{\prp{CEP}}) if for any $\m{A},\m{B}\in\K$ and any congruence $\Theta$ of $\m{A}$, if $\m{A}$ is a subalgebra of $\m{B}$, then there exists a congruence $\Phi$ of $\m{B}$ such that $\Phi\cap A^2 = \Theta$.  An algebra $\m{A}$ is \emph{simple} if its only congruences are the identity congruence $\Delta_\m{A} = \{(x,x)\mid x\in A\}$ and the full congruence $\nabla_\m{A} = A\times A$.

An algebra $\m{A}$ is called \emph{finitely subdirectly irreducible} if the identity congruence of $\m{A}$ is meet irreducible in the lattice of congruences of $\m{A}$, i.e., whenever $\Delta_\m{A}=\Theta\cap\Phi$ for congruences $\Theta$ and $\Phi$ of $\m{A}$, either $\Delta_\m{A}=\Theta$ or $\Delta_\m{A}=\Phi$. When $\V$ is a variety, we denote by $\Vfsi$ the class of finitely subdirectly irreducible algebras in $\V$.

It is well known that a variety $\V$ has the \prp{TIP} if and only if it has both the \prp{CEP} and the \prp{AP}; see, e.g., \cite{Fussner2024}. Further, if $\lgc{L}$ is an algebraizable logic and $\V$ is a variety comprising its equivalent algebraic semantics then, by \cite[Theorem 2.2]{Czel1985}, $\lgc{L}$ has the \prp{MIP} if and only if $\V$ has the \prp{TIP}.

For any class $\K$ of similar algebras, we write $\hm(\K)$ for the class of homomorphic images of algebras in $\K$, $\sub(\K)$ for the class of subalgebras of algebras in $\K$, and $\pu(\K)$ for the class of ultraproducts of members of $\K$. Finally, the variety generated by a class $\K$ of similar algebras is denoted $\vr(\K)$.\\

\paragraph{\bf The main argument.} We turn to the proof of our main theorem. We rely on the fact that $\lgc{R}$ (without constants) has as its equivalent algebraic semantics the variety $\cls{R}$ of \emph{relevant algebras}; see \cite{FontRod1990}. We consider, in particular, the crystal lattice $\m{C}$ which is specified by the labeled Hasse diagram given in Figure~\ref{fig:C}. The variety generated by $\m{C}$ is denoted by $\C$, and the schematic extension of $\lgc{R}$ corresponding to this variety is denoted by $\LC$.\footnote{An axiomatic presentation of $\LC$ can be found in \cite[p.~114]{RLR2}.}

\begin{figure}
\centering
\begin{tikzpicture}[
place/.style={circle,draw=black,fill=black, minimum size = 4pt, inner sep = 0pt},
square/.style={regular polygon,regular polygon sides=4},
place2/.style={square,draw=black,fill=black, minimum size = 5.5pt, inner sep = 0pt},
place3/.style={square,draw=black, minimum size = 5.5pt, inner sep = 0pt},
place4/.style={circle,draw=black, minimum size = 4pt, inner sep = 0pt}]

   \node[place] (top) at (0,-1.6) {};
   \node[place4] (tm) at (0,-2.4) {};
  \node[place] (ml) at (-.8,-3.2) {};
    \node[place] (mr) at (0.8,-3.2) {};
  \node[place] (bm) at (0,-4) {};
  \node[place] (bot) at (0,-4.8) {}; 
   
  \node[left] () at (top) {$ab=\zr^2=\top$\;\;};
  \node[left] () at (tm) {$\zr$\;\;};
  \node[left] () at (ml) {$\neg a = a^2 = a$\;};
  \node[right] () at (mr) {\;$b = b^2 = \neg b$};
  \node[left] () at (bm) {$\neg\zr=\tc$\;\;};
  \node[left] () at (bot) {$\neg\top=\bot$\;\;};
  
  \draw (bot) -- (bm) -- (mr) -- (bm) -- (ml) -- (tm) -- (mr) -- (tm) -- (top);
\end{tikzpicture}
\caption{Labeled Hasse diagram for the relevant algebra $\m{C}$. Here we follow the standard convention that idempotent elements are denoted with filled in nodes, whereas non-idempotent elements are denoted with nodes that are not filled in.}
\label{fig:C}
\end{figure}
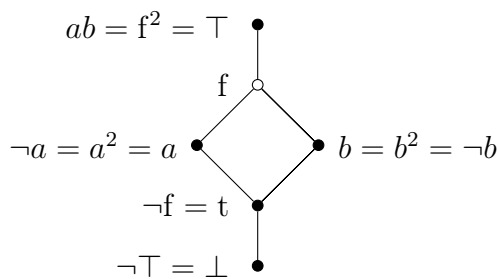

\begin{lemma}\label{lem:CEP}
$\C$ has the congruence extension property.
\end{lemma}

\begin{proof}
By \cite[Theorem~2.3]{Fussner2024}, a congruence-distributive variety $\V$ has the \prp{CEP} if and only if its class of finitely subdirectly irreducible members does.

Note that $\C$ is congruence distributive since relevant algebras have lattice reducts. Moreover, by J\'{o}nsson's Lemma (see, e.g., \cite[Section~1.2.6]{GalatosJipsenKowalskiOno2007}), $\C_\mathrm{FSI}\subseteq\hspu(\m{C}) = \hm\sub(\m{C})$. It suffices to show that $\C_{\mathrm{FSI}}$ has the \prp{CEP}. Direct computation shows that the proper subalgebras of $\C$ have universes
\begin{align*}
& \{a\}, \{b\},\\
& \{\top,\bot\}, \\
& \{\bot,a,\top\}, \{\bot,b,\top\},\\
& \{\bot,\tc,\zr,\top\},\\
& \{\bot,\tc,a,\zr,\top\},\text{ and } \{\bot,\tc,b,\zr,\top\}.
\end{align*}
One may show that all of these are simple by direct computation. For example, suppose there were a non-trivial congruence $\Theta$ on $\{\bot,\tc,\zr,\top\}$ other than the identity. Note that if $\langle\zr,\tc\rangle\in\Theta$, then $\Theta$ is the trivial congruence, as then $\langle\zr\rightarrow\tc,\zr\rightarrow\zr\rangle\in\Theta$ (where $x\rightarrow y=\neg(x\cdot\neg y)$) and so $\langle\bot,\tc\rangle\in\Theta$. It follows that $\langle\top,\zr\rangle\in\Theta$, so $\Theta$ is trivial. But every non-trivial congruence on $\{\bot,\tc,\zr,\top\}$ identifies $\tc$ and $\zr$: For example, if $\langle\top,\tc\rangle\in\Theta$ then $\langle\top\land\zr,\top\land\tc\rangle=\langle\zr,\tc\rangle\in\Theta$. Thus $\{\bot,\tc,\zr,\top\}$ is simple. Let's consider also $\{\bot,\tc,a,\zr,\top\}$. If $\langle a,\tc\rangle\in\Theta$, then $\langle a\rightarrow\tc,\tc\rightarrow\tc\rangle=\langle\bot,\tc\rangle\in\Theta$, and so $\Theta$ is trivial. Suppose that $\langle a,\top\rangle\in\Theta$. Then $\langle\zr\land\top,\zr\land a\rangle=\langle\zr,a\rangle\in\Theta$, and it follows that $\Theta$ is trivial. Similar computations may be performed in the remaining cases.
\end{proof}

For our main theorem, we need a suitable characterization of the \prp{AP}, supplied in the next lemma. For this, we say that an algebra $\m{A}$ is \emph{extensible} if whenever $\m{A}_1,\m{A}_2$ are non-trivial subalgebras of $\m{A}$ and $\varphi\colon\m{A}_1\bij\m{A}_2$ is an isomorphism, there is an automorphism $\hat{\varphi}\colon\m{A}\bij\m{A}$ of $\m{A}$ extending $\varphi$. Note that the lemma follows directly from \cite[Corollary~3.16]{Kearnes1991}. A different proof may be given by using \cite[Corollary~3.5]{Fussner2024}, and a third proof, using the results of \cite{FSbasic,FSsemilin}, is given in \cite{SantschiThesis}.

\begin{lemma}\label{lem:extensible}
Let $\m{A}$ be any finite simple algebra with the congruence extension property, and suppose that $\V=\vr(\m{A})$ is congruence distributive. Then $\V$ has the amalgamation property if and only if $\m{A}$ is extensible.
\end{lemma}

\begin{thm}
The logic $\lgc{C}$ has both the variable sharing property and the Maehara interpolation property.
\end{thm}

\begin{proof}
We have already noted that $\lgc{C}$ has the \prp{VSP}, so it suffices to show that $\C$ has the \prp{TIP}, which holds if and only if $\C$ has the both \prp{AP} and the \prp{CEP}. That $\C$ has the \prp{CEP} follows from Lemma~\ref{lem:CEP}. It is enough that $\C$ has the \prp{AP} and, by Lemma~\ref{lem:extensible}, we need only show that $\m{C}$ is extensible.

Let $\m{A}$ and $\m{B}$ be distinct, non-trivial isomorphic subalgebras of $\C$. These are already listed in the proof of Lemma~\ref{lem:CEP}. Without loss of generality, we may take $a\in\m{A}$, $b\in\m{B}$, and further assume that the isomorphism $\varphi\colon\m{A}\bij\m{B}$ between them satisfies $\varphi(a)=b$. Any such isomorphism is extended by the automorphism $\hat{\varphi}\colon\C\bij\C$ defined by
\[ \hat{\varphi}(x)=\begin{cases} 
      b & x= a \\
      a & x= b \\
      x & otherwise.
   \end{cases}
\]
Thus $\C$ is extensible, and the result follows.
\end{proof}

Two closing remarks are in order. Firstly, while we have seen that $\LC$ has both the \prp{VSP} and \prp{MIP}, other ways of construing `relevance' and `interpolation' are, of course, possible. Rather than the \prp{VSP}, one may take, for example, the Basic Relevance Property of Avron as in \cite{Avron2014}. Instead of the \prp{MIP}, one might take the Craig interpolation property formulated in terms of $\to$.\footnote{That $\LC$ lacks Craig interpolation follows already from Urquhart's result in \cite[Sec. 11.6]{RLR2}, which uses $\m{C}$ to provide the counterexample.} The combinations between various relevance and interpolation properties are numerous. Rather than catalog all such interactions, we seek here only to show that relevance and interpolation are compatible under \emph{some} given readings. We hope that this note stimulates a fuller inquiry into \emph{how} various relevance properties and interpolation properties interact.

Secondly, and in suggesting directions of the aforementioned inquiry, we note that the other maximal extension of $\lgc{R}$ with the \textsc{VSP} does not have the \prp{MIP}, since the variety $\M$ generated by Belnap's model $\m{M}$ does not have the \prp{CEP}; see \cite[p.~289]{Czel1999}. Indeed, it follows from \cite[Corollary~2.11]{Kearnes1989} that $\M$ does not have the \prp{AP}, and hence does not even validate the Robinson consistency theorem. However, this is not sufficient to determine whether this logic has the deductive interpolation property, which remains an open question.

\end{document}